\RequirePackage{fix-cm}
\documentclass[smallextended]{svjour3}       
\smartqed  
\usepackage{graphicx}
\usepackage{subfigure}
\usepackage{enumerate}
\usepackage{xcolor}
\usepackage{amsmath}
\usepackage{amssymb,epsfig,multirow}
\usepackage{float}
\usepackage{color}
\usepackage{booktabs}
\usepackage{enumerate}

\begin{document}
\allowdisplaybreaks
\title{Closing the duality gap of the generalized trace ratio problem
\thanks{This research was supported by the National Natural Science Foundation of China under grants 12101041, 12171021.}
}


\author{ Meijia Yang \and   Yong Xia    }


\institute{
M. Yang \at School of Mathematics and Physics, University of Science and Technology Beijing, Beijing, 100083, P. R. China.\\
Y. Xia \at LMIB of the Ministry of Education, School of Mathematical Sciences, Beihang University, Beijing, 100191, P. R. China.\\
\email{myang@ustb.edu.cn (M. Yang); yxia@buaa.edu.cn (Y. Xia, Corresponding author)}
 }

\date{Received: date / Accepted: date}

\maketitle

\begin{abstract}
The generalized trace ratio problem {\rm (GTRP)} is to maximize a quadratic fractional objective function in trace formulation over the Stiefel manifold. 
In this paper, based on a newly developed matrix S-lemma, we show that {\rm (GTRP)}, if a redundant constraint is added and well scaled, has zero Lagrangian duality gap. However, this is not always true without the technique of scaling or adding the redundant constraint.  

 \keywords{Trace ratio problem \and Fractional programming \and Stiefel manifold \and Matrix S-lemma \and Lagrangian duality}
\subclass{90C26 \and  90C32 \and 90C46}
\end{abstract}

\section{Introduction}
Consider the generalized trace ratio problem {\rm (GTRP)}: 
\begin{eqnarray*}
{\rm(GTRP)}~&\max& \Phi(X)=\frac{tr(GX^TBX)}{tr(GX^TAX)}\\
~~~~~~~~~~~&{\rm s.t.}& X^TX=I_p,
\end{eqnarray*} 
where $X\in\Bbb R^{n\times p}$, $n\geq p$, $B\in\Bbb S^{n\times n}$ means that $B$ is a symmetric matrix in the real space with dimension $n\times n$. Throughout this paper, we assume that both $G\in\Bbb S^{p\times p}$ and $A\in\Bbb S^{n\times n}$ are positive definite matrices so that $tr(GX^TAX)> 0$ always holds for $\forall X\neq 0$. 

{\rm (GTRP)} is presented in \cite{YZX2023}, where the authors established the first-order necessary and sufficient optimality conditions for the global maximizer of {\rm (GTRP)}, the authors also proved that {\rm (GTRP)} has no local non-global maximizer and revealed the hidden convexity of {\rm (GTRP)}.  

{\rm (GTRP)} plays an important role in pattern recognition, computer vision and machine learning \cite{SDH2010}. {\rm (GTRP)} contains many famous problems as its special cases. For example, 
the problem of maximizing $\Phi(X)$ with $X$ being a permutation matrix is a fractional quadratic assignment problem {\rm (FQAP)}, where $G$ denotes the distance matrix, $A$ and $B$ denote the flow matrix. Then {\rm (GTRP)} can be seen as the orthogonal relaxation of {\rm (FQAP)}. In {\rm (GTRP)}, when $G=I_p$, it reduces to the trace ratio problem (TRP), which is well studied in \cite{ZLN2010,ZYL2014}; when $n=p$ and $A=I_n$ (or $tr(GX^TAX)\equiv 1$), it reduces to the orthogonal constrained homogeneous quadratic programming problem, denoted by {\rm (GP)}. {\rm (GP)} is well studied in \cite{ACWY1999,AW1999}; when $A=I_n$ (or $tr(GX^TAX)\equiv 1$), it reduces to the Brockett cost function optimization, which is an orthogonal relaxation of the Koopmans-Beckmann quadratic assignment problem\cite{A2003,Ed1980,Ed1977}, we denote this case by {\rm (GTP)}, which is well studied in \cite{FN2006}; when $G=I_p,~A=I_n$ holds together, it reduces to $\max_{X^TX=I_p} tr(X^TBX)$, whose optimal value is the sum of the first $p$ largest eigenvalues of $B$ \cite{FN2006}; when $p=1$, it reduces to $\max_{x^Tx=1} \frac{x^TBx}{x^TAx}$, which is a generalized Rayleigh quotient problem, denoted by {\rm (GRQ1)}; when $p=1$ and $A=I_n$, it reduces to $\max_{x^Tx=1} \frac{x^TBx}{x^Tx}$, denoted by {\rm (RQ)}, in this case, $x\in\Bbb R^n$ is a vector, so {\rm (RQ)} is the classical Rayleigh quotient problem, whose optimal value is the largest eigenvalue of $B$ \cite{HJ1985}.

Lagrangian dual plays an important role in optimization. If there is no Lagrangian duality gap, we can turn to solve the dual problem. Usually we can obtain the Lagrangian dual from the primal problem directly. But sometimes there might be duality gap. For example, as a special case of {\rm (GTRP)}, {\rm (GRQ1)} enjoys the strong Lagrangian duality property, which we will give an explanation in Example \ref{ex1}. However, in this paper, we will show that for the general case of {\rm (GTRP)}, there might be a Lagrangian duality gap. 

There are two useful techniques in closing the duality gap of some nonconvex optimization problems \cite{Xia2020}. The first one is adding redundant constraint. For example, in \cite{ACWY1999,AW1999}, the authors studied {\rm (GP)}, which is {\rm (GTRP)} with $n=p$ and $A=I_n$ (or $tr(GX^TAX)\equiv 1$). The authors showed that strong duality holds for {\rm (GP)} if the seemingly redundant constraints $XX^T=I_n$ (\cite{AW1999}) or $XX^T\preceq I_n$ (\cite{ACWY1999}) is added before taking the Lagrangian dual. 
In \cite{FN2006}, the authors studied {\rm (GTP)}, which is {\rm (GTRP)} with $A=I_n$. The authors showed that strong duality holds for {\rm (GTP)} if the seemingly redundant constraints $XX^T\preceq I_n$ is added before taking the Lagrangian dual. (The expression $XX^T\preceq I_n$ means that $XX^T-I_n$ is negative semidefinite). So we add $XX^T\preceq I_n$ to {\rm (GTRP)} and obtain {\rm (GR)}:
\begin{eqnarray*}
{\rm (GR)}~&\max& \frac{tr(GX^TBX)}{tr(GX^TAX)}\\
~~~~~~~~~~~&{\rm s.t.}& X^TX=I_p,\\
~~~~~~~~~~~&~~~~~~~~~~& XX^T\preceq I_n.
\end{eqnarray*}
However, we will prove in this paper that for the general case of {\rm (GR)}, it does not enjoy the strong Lagrangian duality property.

The second skill is scaling, which is very useful in fractional programming. In  \cite{YX2020}, the authors studied the two-sided quadratic constraint quadratic fractional programming problem, denoted by {\rm (QF)}, in which if the two-sided constraint reduces to the equality constraint and the quadratic terms are homogeneous, it reduces to {\rm (GRQ1)}, a special case of {\rm (GTRP)} by letting $p=1$. The authors showed that {\rm (QF)}, if well scaled, has strong Lagrangian duality property, which inspires us to make a scale on {\rm (GTRP)} and obtain {\rm (GS)}:
\begin{eqnarray*}
{\rm(GS)}~&\sup\limits_{X\neq 0}& \frac{tr(GX^TBX)}{tr(GX^TAX)}\\
~~~~~~~~~~~&{\rm s.t.}& \frac{X^TX}{tr(GX^TAX)}=\frac{I_p}{tr(GX^TAX)}.
\end{eqnarray*} 
However, in this paper, we will show that for the general case of {\rm (GS)}, it does not enjoy the strong Lagrangian duality property.

Since for the general case of {\rm (GTRP)}, {\rm (GR)}, {\rm (GS)}, all of them have the Lagrangian duality gap. So we combine the skills together and obtain {\rm (GRS)}: 
\begin{eqnarray*}
{\rm(GRS)}~&\sup\limits_{X\neq 0}& \frac{tr(GX^TBX)}{tr(GX^TAX)}\\
~~~~~~~~~~~&{\rm s.t.}& \frac{X^TX}{tr(GX^TAX)}=\frac{I_p}{tr(GX^TAX)},\\
~~~~~~~~~~~&~~~~~~~~~~& \frac{XX^T}{tr(GX^TAX)}\preceq \frac{I_n}{tr(GX^TAX)},
\end{eqnarray*}
we will prove in this paper that {\rm (GRS)} has no Lagrangian duality gap.

The paper is organized as follows. In section 2, we establish a matrix S-lemma based on which we show that {\rm (GRS)} enjoys the strong Lagrangian duality property, which means that we close the duality gap of {\rm (GTRP)} if the seemingly redundant constraint is added and well scaled before we take the Lagrangian dual. In section 3, we show that all of {\rm (GTRP)}, {\rm (GR)}, {\rm (GS)} have duality gap if we take the Lagrangian dual directly (without any other equivalent transformation). Conclusion are made in Section 4.

\textbf{Notation.} Let $x\in \Bbb R^n$ be a vector with dimension $n$ in the real space, and $Diag(x)$ denotes a diagonal matrix with dimension $n\times n$, whose diagonal elements are composed of $x$. Let $I_p$ be the $p\times p$ dimensional identity matrix. Let $A\in\Bbb S^{n\times n}$ be a symmetric matrix in the real space with dimension $n\times n$. Let $\lambda_{\max}(A)$ be the largest eigenvalue of the matrix $A$. Let $v(\cdot)$ be the optimal value of $(\cdot)$. We use $A\preceq(\succeq) B$ (or equivalently $A-B\preceq(\succeq) 0$) to denote that the matrix $A-B$ is negative(positive) semi-definite. Let $X\in\Bbb R^{n\times p}$ be a matrix with dimension $n\times p$ in the real space. We use vec(X) to denote the vector by stacking the columns of the matrix $X$. Denote the Kronecker product by $\otimes$. In order to improve the readability, we review the following three properties about Kronecker product \cite{HJ1985}:\\
(a). $tr(GX^TAX)=vec(X)^T(G\otimes A)vec(X)$,\\
$~~~~~~~tr(GXX^T)=vec(X)^T(I_p\otimes G)vec(X).$\\
(b). $(G\otimes A)^{-1}=G^{-1}\otimes A^{-1}$.\\
(c). $(A\otimes C)(B\otimes D)=AB\otimes CD$.
\section{Strong Lagrangian duality for {\rm (GRS)}}
In this section, we show that there is no Lagrangian duality gap for {\rm (GRS)}, which is equivalent to (GTRP). 
We first review an extension of the Farkas lemma, which will be used in the following proof. The Farkas lemma \cite{F1902} is the well-known alternative theorem for linear system. It has been extended to the convex system, which is known as convex Farkas lemma \cite{DJ2014,J2000,KRT2004,R70,SW1970,YWX2022}. 
\begin{lemma}[Convex Farkas lemma]\label{Flem}
Let $f, g_1, \ldots, g_m: \Bbb R^n \rightarrow \Bbb R$ be convex functions and $\Omega$ be a convex set in $\Bbb R^n$. Assume Slater's condition holds for $g_1, \ldots, g_m$,
i.e., there exists an $\widetilde{x}\in relint(\Omega)$ (the set of relative interior points of $\Omega$) such that for $i = 1, \ldots,m$, $g_i(\widetilde{x}) < 0$ if $g_i(x)$ is nonlinear, and $g_i(\widetilde{x}) \leq 0$
if $g_i(x)$ is linear. The following two
statements are equivalent:
\begin{itemize}
\item[(i)] The system $\{f(x)<0,~g_i(x)\le 0, ~i=1,\ldots,m,~x\in \Omega\}$ is not solvable.
\item[(ii)] There exist $\lambda_i\ge 0~(i=1,\ldots,m)$ such that
\[
f(x)+\sum_{i=1}^m \lambda_ig_i(x)\ge 0,~\forall x\in \Omega.
\]
\end{itemize}
\end{lemma}

Based on Lemma \ref{Flem}, we obtain the following special case of convex Farkas lemma, which is named as nonhomogeneous Farkas lemma and will be used in the proof of the main theorem (matrix S-lemma). We give a proof of nonhomogeneous Farkas lemma to improve the readability. 
\begin{lemma}[Nonhomogeneous Farkas Lemma]\label{lemmaFarkas}
Let $\tilde{A}\in\Bbb R^{p\times q}$, $\tilde{b}\in \Bbb R^p$,~$\tilde{c}\in \Bbb R^q$. Then the following two statements are equivalent:
\begin{itemize}
\item[(i)] The linear system $\tilde{A}^T\tilde{y}\leq \tilde{c},~\tilde{b}^T\tilde{y}< 0,~\tilde{y}(\in\Bbb R^p)\geq 0$ has no solution.
\item[(ii)]There exists $\tilde{x}(\in\Bbb R^q)\geq0$ such that $\tilde{A}\tilde{x}+\tilde{b}\geq 0$ and $\tilde{c}^T\tilde{x}\leq 0$.
\end{itemize}
\end{lemma}
\begin{proof}
Let $\tilde{A}\in\Bbb R^{p\times q}$, $\tilde{b}\in \Bbb R^p$,~$\tilde{c}\in \Bbb R^q$, $\tilde{y}\in\Bbb R^p$. In (i) of Lemma \ref{Flem}, let $f(\tilde{y})=\tilde{b}^T\tilde{y}$, $g_{i}(\tilde{y})=-\tilde{y}_i,~i=1,...,p$, $g_{p+1}(\tilde{y})=\tilde{A}^T\tilde{y}-\tilde{c}$, $\Omega=\Bbb R^n$. Then 
the system 
$\{f(\tilde{y})<0,~g_i(\tilde{y})\le 0, ~i=1,\ldots,p+1\}$ is not solvable is equivalent to the following statement that there exist $\tilde{\lambda}(\in\Bbb R^p)\geq 0$ and $\tilde{x}(\in\Bbb R^p)\geq 0$ such that 
$
\tilde{b}^T\tilde{y}-\sum_{i=1}^p\tilde{\lambda}_i\tilde{y}_i+\tilde{x}^T(\tilde{A}^T\tilde{y}-\tilde{c})\geq 0,~\forall \tilde{y}\in\Bbb R^p
$
hold. It can be further reformulated as 
$
(\tilde{b}-\tilde{\lambda}+\tilde{A}\tilde{x})^T\tilde{y}-\tilde{c}^T\tilde{x}\geq 0,~\forall \tilde{y}\in\Bbb R^p
$, 
which is equivalent to the system
$
\{\tilde{b}-\tilde{\lambda}+\tilde{A}\tilde{x}=0,~\tilde{c}^T\tilde{x}\leq 0\}.
$
Since $\tilde{\lambda}\geq 0$, the above system can be further equivalent to 
$
\{\tilde{b}+\tilde{A}\tilde{x}\geq0,~\tilde{c}^T\tilde{x}\leq 0\}
$. The proof is complete.
\end{proof}

We provide the following lemma for completeness.
\begin{lemma}\label{lemma1}
  Let $X\in\Bbb R^{n\times p}$ with $p\leq n$, then $X^TX=I_p$ implies $XX^T\preceq I_n$.
\end{lemma}
\begin{proof}
  Since $X^TX=I_p$, then $XX^T=X(X^TX)X^T=XX^TXX^T$. Let $Y=XX^T$, then we have $Y=Y^2$ (so $Y$ is an idempotent matrix). 
  
  Let $\lambda$ be any eigenvalue of $Y$ and the corresponding eigenvector is $x(\in\Bbb R^n)\neq 0$, then according to $Yx=\lambda x$, we have $Y^2x=Y\lambda x=\lambda^2x$. Combined with $Y=Y^2$, we obtain $\lambda x=\lambda^2x$, hence $\lambda=0~or~1$ since $x\neq 0$. So the eigenvalues of $Y$ are either 0 or 1. 
  
  So $0_n\preceq Y\preceq I_n$, i.e., $(0_n\preceq)XX^T\preceq I_n$. The proof is complete.
\end{proof}

\begin{lemma}\label{lemma4}
Let $H(\in\Bbb S^{p\times p})=U_hDiag(\lambda)U_h^T$ and $Q(\in\Bbb S^{n\times n})=U_qDiag(\mu)U_q^T$, where $U_h$ and $U_q$ are orthogonal matrices due to the symmetry of $H$ and $Q$, $\lambda=(\lambda_1,...,\lambda_p)^T$ and $\mu=(\mu_1,...,\mu_n)^T$ are the vectors composed of the eigenvalues of $H$ and $Q$, respectively. Let  
\begin{eqnarray*}
S_1&=&\left\{X\in\Bbb R^{n\times p}:~tr(HX^TQX)<0,X^TX=I_p\right\},\label{eqM1}\\
S_2&=&\left\{\pm U_q\sqrt{Z}U_h^T:~Z\in\Bbb R^{n\times p},\sum_{i=1}^n\sum_{j=1}^p\lambda_j\mu_iZ_{ij}<0,\sum_{i=1}^nZ_{ij}=1,\forall j=1,...,p;\right.\nonumber\\
&&\left.~~~~~~~~~~~~\sum_{j=1}^pZ_{ij}\leq1,\forall i=1,...,n;Z_{ij}\geq 0,i=1,...,n,j=1,...,p\right\}.\label{eqM4}
\end{eqnarray*} 
Then the following two results hold true.\\
(R1) $S_1 \subseteq S_2$.\\
(R2) The statement $S_1=\emptyset$ is equivalent to $S_2=\emptyset$.
\end{lemma}
\begin{proof}
\textbf{Proof of (R1)}. According to the definition, we obtain that 
\begin{eqnarray}
&&S_1\nonumber\\
&=&\left\{X\in\Bbb R^{n\times p}:~tr(U_hDiag(\lambda)U_h^TX^TU_qDiag(\mu)U_q^TX)<0,~X^TX=I_p\right\}\nonumber\label{eqM5}\\
&=&\left\{X\in\Bbb R^{n\times p}:~tr(Diag(\lambda)U_h^TX^TU_qDiag(\mu)U_q^TXU_h)<0,~X^TX=I_p\right\}\nonumber\\
&=&\left\{U_qYU_h^T\in\Bbb R^{n\times p}:tr(Diag(\lambda)Y^TDiag(\mu)Y)<0,Y^TY=I_p,Y\in\Bbb R^{n\times p}\right\}\label{eqM2}\\
&=&\left\{U_qYU_h^T\in\Bbb R^{n\times p}:tr(Diag(\lambda)Y^TDiag(\mu)Y)<0,Y^TY=I_p,YY^T\preceq I_n\right\}\label{eqM3}\\
&\subseteq &\left\{U_qYU_h^T\in\Bbb R^{n\times p}:\sum_{i=1}^n\sum_{j=1}^p\lambda_j\mu_iY_{ij}^2<0,\sum_{i=1}^nY_{ij}^2=1,\forall j=1,...,p;\right.\nonumber\\
&&\left.~~~~~~~~~~~~~~~~~~~~~~~~~~~~~~~~~~~~~~~~~~~~~~~~~~~~~~~~~\sum_{j=1}^pY_{ij}^2\leq1,\forall i=1,...,n\right\}\label{ieqM1}\\
&=&S_2,\label{eqM4}
\end{eqnarray}    
where (\ref{eqM2}) holds by letting $Y=U_q^TXU_h$, (\ref{eqM3}) holds by Lemma \ref{lemma1}, (\ref{ieqM1}) holds by relaxing the off-diagonal constraints in $Y^TY=I_p$ and $YY^T\preceq I_n$. (\ref{eqM4}) holds by letting $Z_{ij}=Y_{ij}^2$ for $\forall i=1,...,n,~j=1,...,p$, and $\sqrt{Z}$ means that we take the square root of each element of the matrix $Z$. Thus $S_1\subseteq S_2$. The proof of (R1) is complete.

\textbf{Proof of (R2)}. Since $S_1\subseteq S_2$, so if $S_2=\emptyset$, it must hold that $S_1=\emptyset$. Then we mainly prove that if $S_1=\emptyset$, then $S_2=\emptyset$ holds. We prove this case on the contrary, which means that we prove if $S_2\neq\emptyset$, then $S_1\neq\emptyset$. 
In $S_2$, we first introduce $n\times(n-p)$ nonnegative slack 
variables $\tilde{Z}_{ij},~i=1,...,n,~j=p+1,...,n$ and obtain $S_2\Leftrightarrow S_5$, where  
\begin{eqnarray}
&&S_5
= \left\{\pm U_q\sqrt{Z}U_h^T:\tilde{Z}\in\Bbb R^{n\times n},\sum_{i=1}^n\sum_{j=1}^p\lambda_j\mu_i\tilde{Z}_{ij}<0,
\sum_{i=1}^n\tilde{Z}_{ij}=1,\right.\nonumber\\
&&\left.\forall j=1,...,n;
\sum_{j=1}^n\tilde{Z}_{ij}=1,\forall i=1,...n;\tilde{Z}_{ij}\geq 0,i=1,...,n,j=1,...,n\right\},\label{eq51}
\end{eqnarray}
where $\tilde{Z}(\in\Bbb R^{n\times n}):= \left(Z,~\bar{Z}\right)$ with
$Z_{ij}$=$\tilde{Z}_{ij}$ for $i=1,...,n,~j=1,...,p$, $\bar{Z}_{ij}=\tilde{Z}_{i,p+j}$ for $i=1,...,n,~j=1,...,n-p$, i.e.,
\begin{eqnarray*}
Z=\left(\begin{array}{cccc}
\tilde{Z}_{11}&\tilde{Z}_{12}&...&\tilde{Z}_{1p}\\
\tilde{Z}_{21}&\tilde{Z}_{22}&...&\tilde{Z}_{2p}\\
\vdots\\
\tilde{Z}_{n1}&\tilde{Z}_{n2}&...&\tilde{Z}_{np}
\end{array}\right),~
\bar{Z}:=\left(\begin{array}{cccc}
\tilde{Z}_{1,p+1}&\tilde{Z}_{1,p+2}&...&\tilde{Z}_{1,n}\\
\tilde{Z}_{2,p+1}&\tilde{Z}_{2,p+2}&...&\tilde{Z}_{2,n}\\
\vdots\\
\tilde{Z}_{n,p+1}&\tilde{Z}_{n,p+2}&...&\tilde{Z}_{n,n}
\end{array}\right).
\end{eqnarray*}
So we turn to prove that if $S_5\neq\emptyset$, then $S_1\neq\emptyset$. It is known that  
\begin{eqnarray*}
S_3&=&\left\{\tilde{Z}\in\Bbb R^{n\times n}:~\sum_{i=1}^n\tilde{Z}_{ij}=1,\forall j=1,...,n;~\sum_{j=1}^n\tilde{Z}_{ij}=1,\forall i=1,...,n;\right.\\
&&\left.~~~~~~~~~~~~~~~~~~~\tilde{Z}_{ij}\geq 0,~\forall i=1,...,n,j=1,...,n\right\}
\end{eqnarray*}
is the set of doubly stochastic matrix, whose vertices are in the following set composed of permutation matrix:
\begin{eqnarray*}
S_4&=&\left\{\tilde{Z}\in\Bbb R^{n\times n}:~\sum_{i=1}^n\tilde{Z}_{ij}=1,\forall j=1,...,n;~\sum_{j=1}^n\tilde{Z}_{ij}=1,\forall i=1,...,n;\right.\\
&&\left.~\tilde{Z}_{ij}\in\{0,~1\},\forall i=1,...,n,j=1,...,n\right\}.
\end{eqnarray*}
Then $S_5=\left\{\pm U_q\sqrt{Z}U_h^T:\sum_{i=1}^n\sum_{j=1}^p\lambda_j\mu_i\tilde{Z}_{ij}<0,\tilde{Z}\in S_3\right\}$. Assume that $S_5\neq \emptyset$ and denote the feasible solution of $S_5$ by $\tilde{Z}^*$. Following the 
the Birkhoff-von Neumann theorem \cite{B1946,vN1953} that any doubly stochastic matrix is the convex combination of permutation matrices, we can say that there exist vertices $\{\tilde{Z}^1,~\tilde{Z}^2,...,\tilde{Z}^K\}\in S_4$ such that 
$\tilde{Z}^*=\sum_{k=1}^Ka_k\tilde{Z}^k$, where $a_k\geq 0,~\forall k=1,...,K,~\sum_{k=1}^Ka_k=1$, 
 and 
\[
0> \sum_{i=1}^n\sum_{j=1}^p\lambda_j\mu_i\tilde{Z}_{ij}^*=\sum_{i=1}^n\sum_{j=1}^p\lambda_j\mu_i(\sum_{k=1}^Ka_k\tilde{Z}^k)_{ij}
 =\sum_{k=1}^K\sum_{i=1}^n\sum_{j=1}^p\lambda_j\mu_i(a_k\tilde{Z}^k)_{ij}
 \] 
 hold. As a result, there must exist $k_0$ such that $\tilde{Z}^{k_0}\in S_4$ and it satisfies
$
\sum_{i=1}^n\sum_{j=1}^p\lambda_j\mu_i(\tilde{Z}^{k_0})_{ij}<0.
$

Up to now, we have proved the existence of $\tilde{Z}^{k_0}\in S_4$. Next we recover $Y^{k_0}$. Since $\tilde{Z}^{k_0}\in S_4$, 
we can conclude that it satisfies
\begin{eqnarray*}
&&\sum_{i=1}^n\tilde{Z}^{k_0}_{ij}=1,\forall j=1,...,p;~\sum_{j=1}^p\tilde{Z}^{k_0}_{ij}\leq1,\forall i=1,...,n;\\
&&\tilde{Z}^{k_0}_{ij}\in\{0,~1\},\forall i=1,...,n,~j=1,...,p.
\end{eqnarray*}
More precisely, for $\sum_{j=1}^p\tilde{Z}^{k_0}_{ij}\leq1,\forall i=1,...,n$, there are $p$ of them that satisfies $\sum_{j=1}^p\tilde{Z}^{k_0}_{ij}=1$, and the other $n-p$ satisfies $\sum_{j=1}^p\tilde{Z}^{k_0}_{ij}=0$.

Let $Y^{k_0}\in\Bbb R^{n\times p}$, and $(Y^{k_0})_{ij}=\sqrt{(Z^{k_0})_{ij}}, i=1,...,n,~j=1,...,p$, then it holds that
$
\sum_{i=1}^n\sum_{j=1}^p\lambda_j\mu_i((Y^{k_0})_{ij})^2<0,~\sum_{i=1}^n(Y^{k_0})_{i}=1,\sum_{j=1}^p(Y^{k_0})_{j}\leq1,~Y^{k_0}\in\{0,1\}.
$
It is obvious that 
\[
Y^{k_0}\in \{Y\in\Bbb R^{n\times p}:~tr(Diag(\lambda)Y^TDiag(\mu)Y)<0,~Y^TY=I_p,~YY^T\preceq I_n\}.
\]

To sum up, we have proved that if $S_5\neq\emptyset$ (i.e., $S_2\neq\emptyset$), there must exist $Y^{k_0}$ that satisfies the constraints in (\ref{eqM3}), which is equal to $S_1$. So we have proved that if $S_2\neq\emptyset$, then $S_1\neq\emptyset$.

According to the above statement, it holds true that $S_1=\emptyset\Leftrightarrow S_2=\emptyset$. 
The proof of (R2) is complete.
\end{proof}

Now we are ready to propose the following alternative theorem. We name it as Matrix S-lemma: 
\begin{theorem}[Matrix S-lemma]\label{thmM}
Let $H\in\Bbb S^{p\times p}$ and $Q\in\Bbb S^{n\times n}$. The following two statements are equivalent:\\
(i)~~$\left\{X\in\Bbb R^{n\times p}:~tr(HX^TQX)<0,~X^TX=I_p\right\}=\emptyset$.\\
(ii) There exist $M\in\Bbb S^{p\times p}$ and $W(\in\Bbb S^{n\times n})\succeq 0$, such that 
\begin{eqnarray}
tr\left(HX^TQX\right)+tr\left(M(X^TX-I_p)\right)+tr\left(W(XX^T-I_n)\right)\geq 0,\forall X\in\Bbb R^{n\times p}.\label{thM2}
\end{eqnarray}
\end{theorem}
\begin{proof}
\textbf{We first prove that (ii) implies (i)}. 
We give this proof on the contrary. Assume that there exists $\tilde{X}\in\Bbb R^{n\times p}$ such that $tr(H\tilde{X}^TQ\tilde{X})<0$ and $\tilde{X}^T\tilde{X}=I_p$ hold. According to Lemma \ref{lemma1}, it also holds that $\tilde{X}\tilde{X}^T-I_n\preceq 0$. As a result, 
\[
 tr\left(H\tilde{X}^TQ\tilde{X}\right)+tr\left(M(\tilde{X}^T\tilde{X}-I_p)\right)+tr\left(W(\tilde{X}\tilde{X}^T-I_n)\right)< 0
\]
holds for $\forall M\in\Bbb S^{p\times p}$ and $\forall W(\in\Bbb S^{n\times n})\succeq 0$. It means that there does not exist $M\in \Bbb R^{p\times p}$ and $W\succeq 0$ such that (\ref{thM2}) holds for $\tilde{X}$, which leads to a contradiction. Hence the assumption does not hold. The first part of the proof is complete.
\textbf{Next, we mainly show that (i) implies (ii)}. In the following, we continue to employ the definition of the matrices $H=U_hDiag(\lambda)U_h^T$, $Q=U_qDiag(\mu)U_q^T$ and the sets $S_1$, $S_2$ as in Lemma \ref{lemma4}. Then, according to Lemma \ref{lemma4}, (i) can be restated as $S_1=\emptyset$. Furthermore, according to (R2) of Lemma \ref{lemma4}, we can turn to prove $S_2=\emptyset$ implies (ii). One more step, we can make a transformation of (ii). In (ii), we have  
\begin{eqnarray}
&&tr\left(HX^TQX\right)+tr\left(M(X^TX-I_p)\right)+tr\left(W(XX^T-I_n)\right)\geq 0,\forall X\in\Bbb R^{n\times p}\nonumber\\
\Leftrightarrow&&vec(X)^T (H\otimes Q +M\otimes I_n+I_p\otimes W)vec(X) \nonumber\\
&&~~~~~~~~~~~~~~~~~~~~~~~~~~~~~~~~~~~~~~~~~~ -tr(MI_p)-tr(WI_n)\geq 0,
\forall X\in\Bbb R^{n\times p}\nonumber\\
\Leftrightarrow&&H\otimes Q +M\otimes I_n+I_p\otimes W\succeq 0,~ tr(MI_p)+tr(WI_n)\leq 0.\label{eqM6}
\end{eqnarray}
\textbf{So we finally turn to prove that if $S_2=\emptyset$, then there exist $M\in\Bbb S^{p\times p}$ and $W(\in\Bbb S^{n\times n})\succeq 0$ such that (\ref{eqM6}) holds}.

We can rewrite the linear system in the constraint of $S_2$ as the format in nonhomogeneous Farkas lemma (Lemma \ref{lemmaFarkas}) by stacking the columns of $Z$ and denote it by $\tilde{y}$:
\[
\tilde{y}^T(\geq 0)=(Z_{11},Z_{21},...,Z_{n1}|Z_{12},Z_{22},...,Z_{n2}|...|Z_{1p},Z_{2p},...,Z_{np})\in\Bbb R^{1\times (n\times p)}.
\]
Let
\begin{eqnarray*}
\tilde{A}^T=\left(\begin{array}{cccc}
~E_1&~E_2&...&~E_p\\
-E_1&-E_2&...&-E_p\\
~I_n&~I_n&...&~I_n
\end{array}\right),~~~
E_i=\left(\begin{array}{cccc}
0&~~0&...&~~0\\
\vdots\\
1&~~1&...&~~1\\
\vdots\\
0&~~0&...&~~0
\end{array}\right),~\forall i=1,...,p,
\end{eqnarray*}
where $\tilde{A}^T\in\Bbb R^{(2p+n)\times(n\times p)}$, and $E_i\in\Bbb R^{p\times n},\forall i=1,...,p$ is the matrix with the entries in the $i$-th row are $1$ and the others are 0.

Let 
\begin{eqnarray*}
&&\tilde{c}^T=\left(1,1,...,1|-1,-1,...,-1|1,1,...,1
\right)\in\Bbb R^{1\times (2p+n)},\\
&&\tilde{b}^T=
(\lambda_1\mu_1,\lambda_1\mu_2,...,\lambda_1\mu_n|\lambda_2\mu_1,\lambda_2\mu_2,...,\lambda_2\mu_n|...|\lambda_p\mu_1,\lambda_p\mu_2,...,\lambda_p\mu_n)\\
&&~~~~\in\Bbb R^{1\times (n\times p)}.
\end{eqnarray*}
Then \{$S_2=\emptyset$\}$\Leftrightarrow$\{$\tilde{b}^T\tilde{y}<0,~\tilde{A}^T\tilde{y}\leq \tilde{c},~\tilde{y}\geq 0$ has no solution\}. According to nonhomogeneous Farkas lemma (Lemma \ref{lemmaFarkas}), there exists 
\[
\tilde{x}^T(\geq 0)=(\tilde{x}_1,\tilde{x}_2,...,\tilde{x}_p|\tilde{x}_{p+1},\tilde{x}_{p+2},...,\tilde{x}_{2p}|\tilde{x}_{2p+1},\tilde{x}_{2p+2},...,\tilde{x}_{2p+n})\in\Bbb R^{1\times (2p+n)}
\]
such that $\tilde{A}\tilde{x}+\tilde{b}\geq 0$ and $\tilde{c}^T\tilde{x}\leq 0$ hold. Then we get the following system which has $np+1$ inequalities:
\begin{eqnarray*}
&&\tilde{x}_i-\tilde{x}_{p+i}+\tilde{x}_{2p+j}+\lambda_i\mu_j\geq 0,~\forall i=1,...,p,~j=1,...,n,\\
&&\tilde{x}_1+...+\tilde{x}_p-\tilde{x}_{p+1}-...-\tilde{x}_{2p}+\tilde{x}_{2p+1}+...+\tilde{x}_{2p+n}\leq 0.
\end{eqnarray*}
%
%
Then we can introduce $p$ free variable $\hat{x}\in\Bbb R^{p\times 1}$ with $\hat{x}_i=\tilde{x}_i-\tilde{x}_{p+i}$ and the above system turns to be
\begin{eqnarray*}
&&\hat{x}_i+\tilde{x}_{2p+j}+\lambda_i\mu_j\geq 0,~\forall i=1,...,p,~j=1,...,n,\\
&&\hat{x}_1+...+\hat{x}_p+\tilde{x}_{2p+1}+...+\tilde{x}_{2p+n}\leq 0,
\end{eqnarray*}
where $\hat{x}_i\in\Bbb R$ with $i=1,...,p$, and $\tilde{x}_i\geq 0$ with $i=2p+1,...,2p+n$.
Then the above system can be rewritten as
\begin{eqnarray}
&&Diag(\lambda)\otimes Diag(\mu)+D_1\otimes I_n+I_p\otimes D_2\succeq 0,\label{eqD2}\\ &&tr(D_1I_p)+tr(D_2I_n)\leq 0,\label{eqD3}
\end{eqnarray} 
where $D_1=Diag(\hat{x}_1,...,\hat{x}_p)$ and $D_2(\succeq 0)=Diag(\tilde{x}_{2p+1},...,\tilde{x}_{2p+n})$ are diagonal matrices. 
Multiplying $U_h\otimes U_q$ on both sides of the left-hand side of (\ref{eqD2}) and $(U_h\otimes U_q)^T$ on both sides of the right-hand side of (\ref{eqD2}). Following the property of Kronecker product,
 we get  
\begin{eqnarray}
(\ref{eqD2})
&\Leftrightarrow&(U_h\otimes U_q)(Diag(\lambda)\otimes Diag(\mu))(U_h\otimes U_q)^T\nonumber\\
&&+(U_h\otimes U_q)(D_1\otimes I_n)(U_h\otimes U_q)^T
+(U_h\otimes U_q)(I_p\otimes D_2)(U_h\otimes U_q)^T\succeq 0\nonumber\\
&\Leftrightarrow&H\otimes Q+\left(U_hD_1U_h^T\right)\otimes I_n+I_p\otimes \left(U_qD_2U_q^T\right)\succeq 0\nonumber\\
&\Leftrightarrow&H\otimes Q+M\otimes I_n+I_p\otimes W\succeq 0\nonumber.
\end{eqnarray}
As for (\ref{eqD3}), it holds that
\begin{eqnarray*}
&&tr(D_1I_p)+tr(D_2I_n)=tr(D_1U_h^TI_pU_h)+tr(D_2U_q^TI_nU_q)\\
&=&tr(U_hD_1U_h^TI_p)+tr(U_qD_2U_q^TI_n)=tr(MI_p)+tr(WI_n).
\end{eqnarray*}
So
$
(\ref{eqD3})\Leftrightarrow tr(MI_p)+tr(WI_n)\leq 0,
$
where $M=U_hD_1U_h^T$ and $W=U_qD_2U_q^T$. Since $U_q$ is the orthogonal matrix, which keeps the positive-semidefiniteness of $D_2$. As a result, $W\succeq 0$ holds. 

Therefore, we have finished the proof that if $S_2=\emptyset$, then there exist $M\in\Bbb S^{p\times p}$ and $W(\in\Bbb S^{n\times n})\succeq 0$ such that (\ref{eqM6}) holds. The whole proof is complete.
\end{proof}

\begin{theorem}
\begin{eqnarray}
v(GTRP)
&&=\inf\left\{\mu:\mu G\otimes A-G\otimes B+M\otimes I_n+I_p\otimes W\succeq 0,\right.\nonumber\\
&&\left. ~~~~~~~~~~~~~~~~~~~~~~~~~~~~~~~~~tr(M)+tr(W)\leq 0,W\succeq 0\right\}\label{vGTRP}
\end{eqnarray}
\end{theorem}
\begin{proof}
\begin{eqnarray}
&&v(GTRP)\nonumber\\
&=&\inf\left\{\mu: \{X\in\Bbb R^{n\times p}:~\frac{tr(GX^TBX)}{tr(GX^TAX)}>\mu,~X^TX=I_p\}=\emptyset\right\}\nonumber\\
&=&\inf\left\{\mu: \left\{X\in\Bbb R^{n\times p}:~tr(GX^T(\mu A-B)X)<0,~X^TX=I_p\right\}=\emptyset\right\}\nonumber\\
&=&\inf\left\{\mu:~\exists M\in\Bbb S^{p\times p},~W(\in\Bbb S^{n\times n})\succeq 0~s.t.~\mu tr(GX^TAX)-tr(GX^TBX)\right.\nonumber\\
&&~~~~~~~~~~\left.+tr(M(X^TX-I_p))+tr(W(XX^T-I_n))\geq 0,~\forall~X\in\Bbb R^{n\times p} \right\}\label{eq21}\\
&=&\inf\left\{\mu:~\exists M\in\Bbb S^{p\times p},~W(\in\Bbb S^{n\times n})\succeq 0~s.t.~-tr(MI_p)-tr(WI_n)
\right.\nonumber\\
&&+vec(X)^T\left(\mu G\otimes A-G\otimes B+M\otimes I_n+I_p\otimes W\right)vec(X)\left.\geq 0,\forall X\in\Bbb R^{n\times p}\right\}\nonumber\\
&=&(\ref{vGTRP}),\nonumber
\end{eqnarray}
which is a semi-definite programming relaxation for {\rm (GTRP)}. 
(\ref{eq21}) holds due to the Matrix S-lemma (Theorem \ref{thmM}). The proof is complete.
\end{proof}

\begin{theorem}
There is no Lagrangian duality gap for {\rm (GRS)}.
\end{theorem}
\begin{proof} 
The Lagrangian dual of (GRS) is:
\begin{eqnarray}
&&\inf_{M,W\succeq 0}\sup_{X\neq 0}\frac{tr(GX^TBX)-tr(M(X^TX-I_p))+tr(W(I_n-XX^T))}{tr(GX^TAX)}\nonumber\\
&=&\inf_{M,W\succeq 0}\sup_{X\neq 0}\frac{vec(X)^T(G\otimes B-M\otimes I_n-I_p\otimes W)vec(X)+tr(MI_p)+tr(WI_n)}{vec(X)^T(G\otimes A)vec(X)}\nonumber\\
&=&\inf_{M,W\succeq 0}\inf_{\mu\in\Bbb R}\left\{\mu:\right.\nonumber\\
&&~~\left.\mu\geq\frac{vec(X)^T(G\otimes B-M\otimes I_n-I_p\otimes W)vec(X)+tr(MI_p)+tr(WI_n)}{vec(X)^T(G\otimes A)vec(X)},\right.\nonumber\\
&&~~\left.\forall X\in\Bbb R^{n\times p},~X\neq 0\right\}\nonumber \\
&=&\inf_{M,W\succeq 0}~~\mu \nonumber\\
&&~~~s.t.~~vec(X)^T(G\otimes B-M\otimes I_n-I_p\otimes W)vec(X)+tr(M)+tr(W)\nonumber\\
&&~~~~~~~~~-\mu vec(X)^T(G\otimes A)vec(X)\leq 0,~~\forall X\in\Bbb R^{n\times p},~X\neq 0\nonumber\\
&=&\inf_{M,W\succeq 0}~~\mu\nonumber\\
&&~~~s.t.~~tr(M)+tr(W)\leq 0 \nonumber\\
&&~~~~~~~~~G\otimes B-M\otimes I_n-I_p\otimes W-\mu G\otimes A\preceq 0,\label{DGS}
\end{eqnarray}
where $M\in\Bbb S^{n\times n}$, $W(\in\Bbb S^{n\times n})\succeq 0$. Compare (\ref{DGS}) with (\ref{vGTRP}), we can see that the primal and dual values for {\rm (GRS)} are equal. The proof is complete.
\end{proof}
\section{A positive Lagrangian duality gap for {\rm (GTRP)}, {\rm (GR)} and {\rm (GS)}}
In this section, we study the Lagrangian duality gap for  {\rm (GTRP)}, {\rm (GR)} and {\rm (GS)}. Before we come to the main result, we recall the following two famous lemmas, which will be used in the following subsections.
\begin{lemma}[Rayleigh Quotient Theorem\cite{HJ1985}]\label{ray}
 Let $Q\in\Bbb S^{n\times n}$ and $\lambda_1\geq \lambda_2\geq...\geq \lambda_n$ be the eigenvalues of $Q$, $u_1,...,u_n$ be orthogonal and such that $Qu_i=\lambda_iu_i,~i=1,...,n$. Then $\lambda_1\geq x^TQx\geq \lambda_n$ for any unit vector $x\in\Bbb R^n$, with equality in the left-hand (respectively, right-hand) if and only if $Qx=\lambda_1x$ (respectively, $Qx=\lambda_nx$); moreover,
 \[
 \lambda_1=\sup_{x\neq 0}\frac{x^TQx}{x^Tx}(=\max_{x^Tx=1}x^TQx),~and~\lambda_{n}=\inf_{x\neq 0}\frac{x^TQx}{x^Tx}(=\min_{x^Tx=1}x^TQx).
 \]
 \end{lemma}
\begin{lemma}[\cite{FN2006,FBR2987,Xia2011}]\label{Xia2011}
Let $H\in\Bbb S^{n\times n}$, $Q\in\Bbb S^{n\times n}$, $\lambda_1\geq\lambda_2\geq...\geq \lambda_n,~\sigma_1\geq\sigma_2\geq...\geq\sigma_n$ be the eigenvalues of $H$ and $Q$, respectively. For any $X\in\Bbb R^{n\times n}$ satisfying $X^TX=I_n$, it follows that 
\begin{eqnarray*}
\sum_{i=1}^{n}\lambda_i\sigma_{n-i}\leq tr(HX^TQX)\leq\sum_{i=1}^{n}\lambda_i\sigma_i.
\end{eqnarray*}
\end{lemma}
Based on Lemma \ref{ray}, we further obtain the following lemma. We give a proof to improve the readability.
\begin{lemma}[\cite{HJ1985}]\label{lemma7}
Let $B\in\Bbb S^{n\times n}$, $G(\in\Bbb S^{p\times p})\succ 0$, $A(\in\Bbb S^{n\times n})\succ 0$, $X\in\Bbb R^{n\times p}$, $X\neq 0$. Define
\[
R(vec(X))=\frac{vec(X)^T(G\otimes B)vec(X)}{vec(X)^T(G\otimes A)vec(X)}.
\]
Then $\sup\limits_{X\in\Bbb R^{n\times p},X\neq 0} R(vec(X))=\lambda_{max}(A^{-1}B)$.
\end{lemma}
\begin{proof}
Let $y=(G\otimes A)^{\frac{1}{2}}vec(X)$, where $(G\otimes A)^{\frac{1}{2}}$ is the square root of the positive definite matrix $G\otimes A$. Then we can insert $vec(X)=(G\otimes A)^{-\frac{1}{2}}y$ into $R(vec(X))$ and obtain
$
R(y)=\frac{y^TSy}{y^Ty},~S=((G\otimes A)^{-\frac{1}{2}})^T(G\otimes B)(G\otimes A)^{-\frac{1}{2}}.
$  

Then according to Lemma \ref{ray}, we know that the maximum value of $R(y)$ is $\lambda_{\max}(S)$. Denote the optimal solution by $\tilde{y}$. It holds that $S\tilde{y}=\lambda_{\max}(S)\tilde{y}$, where $\tilde{y}$ is the eigenvector corresponding to $\lambda_{\max}(S)$.
%

Consider the secular function
\[
((G\otimes A)^{-\frac{1}{2}})^T(G\otimes B)(G\otimes A)^{-\frac{1}{2}}y=\lambda y,
\]
on whose left-hand side multiplied by $(G\otimes A)^{-\frac{1}{2}}$, we obtain
\[
((G\otimes A)^{-1})^T(G\otimes B)(G\otimes A)^{-\frac{1}{2}}y=\lambda(G\otimes A)^{-\frac{1}{2}}y,
\]
Since $vec(X)=(G\otimes A)^{-\frac{1}{2}}y$, the above equation can be further written as 
\[((G\otimes A)^{-1})^T(G\otimes B)vec(X)=\lambda vec(X).\]
 So in order to find the largest eigenvalue of the matrix $S$ and the corresponding eigenvector $\tilde{y}$, we can turn to find the largest eigenvalue of the matrix 
 \[
 ((G\otimes A)^{-1})^T(G\otimes B)= (G^{-1}\otimes A^{-1})(G\otimes B)=(G^{-1}G)\otimes(A^{-1}B)=I_p\otimes(A^{-1}B),
 \] 
 which corresponds to the Rayleigh quotient problem:
 \[
 \sup_{vec(X)\neq 0}\frac{vec(X)^T(I_p\otimes(A^{-1}B))vec(X)}{vec(X)^Tvec(X)},
 \]
whose optimal value is $\lambda_{\max}(I_p\otimes(A^{-1}B))=\lambda_{\max}(A^{-1}B)$ due to Lemma \ref{ray}
. The proof is complete.
\end{proof}
\subsection{A positive Lagrangian duality gap for {\rm (GTRP)} and {\rm (GR)}}
In this subsection, we show that {\rm (GTRP)}, or even {\rm (GR)}, which is obtained by adding a redundant constraint to {\rm (GTRP)}, could have a positive Lagrangian duality gap. Unlike the effect on {\rm (GTP)} \cite{FN2006}, the operation of adding a redundant constraint has no effect in reducing the duality gap at all.
\begin{theorem}\label{thDGTRP}
  The optimal value of the Lagrangian dual of {\rm (GTRP)} is $\lambda_{\max}(A^{-1}B)$.
\end{theorem}
\begin{proof}
The Lagrangian dual of {\rm (GTRP)} is:
\begin{eqnarray}
&&\inf_{M}\sup_{X\neq 0}\frac{tr(GX^TBX)}{tr(GX^TAX)}+trM(X^TX-I_p)\nonumber\\
&=&\inf_{M} 
\begin{aligned}&\left\{\begin{array}{l}     
\sup\limits_{X\neq 0}\frac{tr(GX^TBX)}{tr(GX^TAX)}-tr(MI_p)~~~~M\preceq 0,~~~~~~~~~~~~~~~\\             
+\infty~~~~~~~~~~~~~~~~~~~~~~~~~~~~~~~~M\npreceq 0
\end{array}\right.     
\end{aligned}\label{eq23}\\
&=&\inf_{M\preceq 0}\sup_{X\neq 0}\frac{tr(GX^TBX)}{tr(GX^TAX)}-tr(MI_p)\nonumber\\
&=&\inf_{M\preceq 0}\sup_{X\neq 0}\frac{vec(X)^T(G\otimes B)vec(X)}{vec(X)^T(G\otimes A)vec(X)}-tr(MI_p)\label{eq37}\\
&=&\inf_{M\preceq 0}\lambda_{\max}\left(A^{-1}B\right)-tr(MI_p)\label{eq36}\\
&=&\lambda_{\max}(A^{-1}B).\nonumber
\end{eqnarray}
where (\ref{eq23}) holds due to the fact that $\frac{tr(GX^TBX)}{tr(GX^TAX)}$ is homogeneous with respect to the variable $X$, then $trMX^TX$ approaches zero at the optimal solution. In other words, denote
\begin{eqnarray*}
&&f(X)=\frac{tr(GX^TBX)}{tr(GX^TAX)}+trM(X^TX-I_p),\\
&&g(X)=\frac{tr(GX^TBX)}{tr(GX^TAX)}-tr(MI_p).
\end{eqnarray*}
Assume that $X^*$ is the optimal solution of $\sup_{X\neq 0}g(X)$ with $M\preceq 0$. Then $f(X)\leq g(X)\leq g(X^*)$, where the first inequality holds since $trMX^TX\leq 0$. Let $Y=tX^*,~t\rightarrow 0$. Then $\lim_{t\rightarrow 0}f(Y)=\lim_{t\rightarrow 0}f(tX^*)=g(X^*)$. As a result, $\sup_{X\neq 0}f(X)=g(X^*)$, the equation (\ref{eq23}) holds.
(\ref{eq36}) holds due to Lemma \ref{lemma7}. The proof is complete.
%
%
\end{proof}
\begin{theorem}\label{dGR}
  The optimal value of the Lagrangian dual of {\rm (GR)} is $\lambda_{\max}(A^{-1}B)$.
\end{theorem}
\begin{proof}
The Lagrangian dual of {\rm (GR)} is:
\begin{eqnarray}
&&\inf_{M,W\succeq 0}\sup_{X\neq 0}\frac{tr(GX^TBX)}{tr(GX^TAX)}+trM(X^TX-I_p)+trW(I_n-XX^T)\nonumber\\
&=&\inf_{M,W\succeq 0}\sup_{X\neq 0}\frac{tr(GX^TBX)}{tr(GX^TAX)}\nonumber\\
&&~~~~~~~~~~~~~~~~~+vec(X)^T(M\otimes I_n-I_p\otimes W)vec(X)-trMI_p+trWI_n\nonumber\\
&=&\inf_{M,W\succeq 0} 
\begin{aligned}&\left\{\begin{array}{l}     
\sup\limits_{X\neq 0}\frac{tr(GX^TBX)}{tr(GX^TAX)}-tr(MI_p)+tr(WI_n),~M\otimes I_n-I_p\otimes W\preceq 0,\\             
+\infty~~~~~~~~~~~~~~~~~~~~~~~~~~~~~~~~~~~~~~~~~~~~~~M\otimes I_n-I_p\otimes W\npreceq 0
\end{array}\right.     
\end{aligned}\label{eq22}\\
&=&\inf_{M,~W\succeq 0,~M\otimes I_n-I_p\otimes W\preceq 0}\sup_{X\neq 0}\frac{tr(GX^TBX)}{tr(GX^TAX)}-tr(MI_p)+tr(WI_n)~\nonumber\\
&=&\inf_{M,~W\succeq 0,~M\otimes I_n-I_p\otimes W\preceq 0}\sup_{X\neq 0}\frac{vec(X)^T(G\otimes B)vec(X)}{vec(X)^T(G\otimes A)vec(X)}-tr(MI_p)+tr(WI_n)~\nonumber\\
&=&\inf_{M,~W\succeq 0,~M\otimes I_n-I_p\otimes W\preceq 0}\lambda_{\max}\left(A^{-1}B\right)-tr(MI_p)+tr(WI_n)\label{eq38}\\
&\geq&\inf_{W\succeq 0}\lambda_{\max}\left(A^{-1}B\right)-\frac{p}{n}trW+trW\label{eq33}\\
&=&\lambda_{\max}(A^{-1}B)\label{eq34}
\end{eqnarray}
where (\ref{eq22}) holds due to the same reason as in (\ref{eq23}). (\ref{eq38}) holds due to Lemma \ref{lemma7}. (\ref{eq33}) holds since $M\otimes I_n-I_p\otimes W\preceq 0$ implies that $ntrM-ptrW\leq 0$, then $-trM\geq-(ptrW)/n$, and hence we have $\inf -tr(MI_p)=-(ptrW)/n$. (\ref{eq34}) holds since $p\leq n$ implies that $1-p/n\geq 0$, together with $W\succeq 0$, we have $\inf (1-\frac{p}{n})trW=0$. The Lagrangian dual of {\rm (GR)} achieves $\lambda_{\max}(A^{-1}B)$ when $W=M=0$. The proof is complete.
\end{proof}

We show in the following example that there is no Lagrangian duality gap for a special case of {\rm (GTRP)}.
\begin{example}\label{ex1}
Consider a special case of {\rm (GTRP)} that $p=1$, then we get
\begin{eqnarray*}
{\rm(GRQ1)}~&\max& \frac{x^TBx}{x^TAx}\\
~~~~~~~~~~~&{\rm s.t.}& x^Tx=1.
\end{eqnarray*}
According to Theorem \ref{thDGTRP}, we know that the optimal value of the Lagrangian dual of {\rm (GRQ1)} is $\lambda_{\max}(A^{-1}B)$. 

Next we consider the primal problem {\rm (GRQ1)}.
According to Lemma \ref{lemma7} with $p=1$, 
 we know that the optimal value of the generalized Rayleigh quotient problem $\max\limits_{x\neq 0}~\frac{x^TBx}{x^TAx}$ is $\lambda_{\max}(A^{-1}B)$, and the optimal solution $x^*$ is any eigenvector corresponding to $\lambda_{\max}(A^{-1}B)$. Furthermore, we can normalize $x^*$, i.e., without loss of generality, we can set $x^*$ to be the corresponding unit eigenvector, i.e., $x^{*T}x^*=1$. We can see that $\max\limits_{x\neq 0}~\frac{x^TBx}{x^TAx}$ is equivalent to {\rm (GRQ1)}, which means that $v{\rm (GRQ1)}=\lambda_{\max}(A^{-1}B)$.   
\end{example}

As discussed above, we can conclude that strong Lagrangian duality holds for {\rm (GRQ1)}, i.e., the optimal value of {\rm (GRQ1)} is equal to the optimal value of its Lagrangian dual problem. However, does strong Lagrangian duality hold for the general case of {\rm (GTRP)} or {\rm (GR)}?

\begin{theorem}\label{th6}
There is no Lagrangian duality gap of {\rm (GTRP)} or {\rm (GR)} if and only if any unit eigenvectors corresponding to $\lambda_{\max}(I_p\otimes(A^{-1}B))$, denoted by $vec(\hat{X})$, satisfy $\hat{X}^T\hat{X}=\frac{1}{p}I_p$.
%
\end{theorem}
\begin{proof}
According to Theorem \ref{thDGTRP} and \ref{dGR}, the optimal value of the Lagrangian dual of {\rm (GTRP)} and {\rm (GR)} is $\lambda_{\max}(I_p\otimes A^{-1}B)$, which corresponds to the generalized Rayleigh quotient problem:
\begin{eqnarray*}
&&\sup\limits_{X\in\Bbb R^{n\times p},X\neq 0}\frac{tr(GX^TBX)}{tr(GX^TAX)}
\bigg(=\sup\limits_{X\in\Bbb R^{n\times p},X\neq 0}\frac{vec(X)^T\left(G\otimes B\right)vec(X)}{vec(X)^T\left(G\otimes A\right)vec(X)}\\
&&~~~~~~~~~~~~~~~~~~~~~~~~~~~~~~~~~~~=\lambda_{\max}(I_p\otimes (A^{-1}B))=\lambda_{\max}(A^{-1}B)\bigg).
\end{eqnarray*}
The Lagrangian duality gap of {\rm (GTRP)} or {\rm (GR)} is:
\begin{eqnarray*}
v_{gap}=\left|\max\limits_{X^TX=I_p} \frac{tr(GX^TBX)}{tr(GX^TAX)}-\lambda_{\max}(A^{-1}B)\right|\geq0,
\end{eqnarray*}
where $v_{gap}=0$ holds when \[\max\limits_{X^TX=I_p}\frac{tr(GX^TBX)}{tr(GX^TAX)}=\lambda_{\max}(A^{-1}B).\]  
Denote any unit eigenvector corresponding to $\lambda_{\max}(I_p\otimes A^{-1}B)$ by $vec(\hat{X})$, i.e. $vec(\hat{X})^Tvec(\hat{X})=1(=tr(\hat{X}^T\hat{X}))$. 
Then there exists $\tilde{X}=\sqrt{p}\hat{X}$ satisfying $vec(\tilde{X})^Tvec(\tilde{X})=p$. 
Consequently, as long as $\hat{X}^T\hat{X}=\frac{1}{p}I_p$, (i.e., $\tilde{X}^T\tilde{X}=I_p$), there is no Lagrangian duality gap for {\rm (GTRP)} and {\rm (GR)}. The proof is complete. 
\end{proof}

To sum up, according to Theorem \ref{th6}, we can see that both {\rm (GTRP)}
and {\rm (GR)} may have a duality gap. And according to Theorems \ref{thDGTRP} and \ref{dGR}, we can see that the operation of only adding a redundant constraint to {\rm (GTRP)} (and obtain {\rm (GR)}) has no effect on closing the duality gap at all.

\subsection{An example for {\rm (GS)} with positive Lagrangian duality gap}
In this subsection, we show that {\rm (GS)}, which is a well-scaled version of {\rm (GTRP)}:
\begin{eqnarray*}
{\rm(GS)}~&\max& \frac{tr(GX^TBX)}{tr(GX^TAX)}\\
~~~~~~~~~~~&{\rm s.t.}& \frac{X^TX}{tr(GX^TAX)}=\frac{I_p}{tr(GX^TAX)},
\end{eqnarray*}
also has Lagrangian duality gap by an example. We first derive the Lagrangian dual of {\rm (GS)} and denote it as {\rm (DGS)}:
\begin{eqnarray}
{\rm (DGS)}&&\inf_{S}\sup_{X\neq 0}\frac{tr(GX^TBX)}{tr(GX^TAX)}+\frac{tr(S(X^TX-I_p))}{tr(GX^TAX)}\nonumber\\
&=&\inf_{S}\sup_{X\neq 0}\frac{vec(X)^T(G\otimes B+S\otimes I_n)vec(X)-tr(S)}{vec(X)^T(G\otimes A)vec(X)}\nonumber\\
&=&\inf_{S}\left\{\inf_{\rho\in\Bbb R} \rho: \rho\geq \frac{vec(X)^T(G\otimes B+S\otimes I_n)vec(X)-tr(S)}{vec(X)^T(G\otimes A)vec(X)},\right.\nonumber\\
&&\left.~~~~~~~~~~~~~~~~~~~~~~~~~~~~~~~~~~~~~~~~~~~~~~~~~~~~~~~~~~\forall X\in\Bbb R^{n\times p},~X\neq 0\right\}\nonumber\\
&=&\inf_{S}\left\{\rho: vec(X)^T(G\otimes B+S\otimes I_n-\rho G\otimes A)vec(X)-tr(S)\leq 0,\right.\nonumber\\
&&\left.~~~~~~~~~~~~~~~~~~~~~~~~~~~~~~~~~~~~~~~~~~~~~~~~~~~~~~~~~~\forall X\in\Bbb R^{n\times p},~X\neq 0\right\}\nonumber\\
&=&\inf_{S}~~\rho\nonumber\\
&&~s.t.~tr(S)\geq 0\label{eq42}\\
&&~~~~~~G\otimes B+S\otimes I_n-\rho G\otimes A\preceq 0,\label{eq43}
\end{eqnarray}
where $S\in\Bbb S^{n\times n}$. Next, we show that there is a gap between {\rm (GS)} and its Lagrangian dual {\rm (GRS)} by an example.
\begin{example}
  Let $m=n=2$, and
\[A=
\left(\begin{array}{cc}
1&~~0\\
0&~~1
\end{array}
\right),~~
G=
\left(\begin{array}{cc}
1&~~0\\
0&~~2
\end{array}
\right),~~
B=
\left(\begin{array}{cc}
1&~~0\\
0&~~3
\end{array}
\right). \]
We denote the primal and dual problem of this case by {\rm (GS1)} and {\rm (DGS1)}, respectively. In order to obtain the optimal value of {\rm (GS1)}, we can simplify it into the following reformulation:
\begin{eqnarray*}
{\rm(GS1)}~&\max& \frac{1}{trG}tr(GX^TBX)\\
~~~~~~~~~~~&{\rm s.t.}& X^TX=I_2.
\end{eqnarray*}
Then according to Lemma \ref{Xia2011}, the optimal value of {\rm (GS1)} is 
${\rm v(GS1)}=\frac{7}{3}$. As for the optimal value of the Lagrangian dual problem {\rm (DGS1)}, denote $S=\left(\begin{array}{cc}
s_{11}&s_{12}\\
s_{12}&s_{22}
\end{array}
\right)$ due to the symmetry of $S$. 
We know that 
\[G\otimes B=
\left(\begin{array}{cccc}
1&0&0&0\\
0&3&0&0\\
0&0&2&0\\
0&0&0&6
\end{array}
\right),~~
S\otimes I_n=
\left(\begin{array}{cccc}
s_{11}&0&s_{12}&0\\
0&s_{11}&0&s_{12}\\
s_{12}&0&s_{22}&0\\
0&s_{12}&0&s_{22}
\end{array}
\right),~~
G\otimes A=
\left(\begin{array}{cccc}
1&0&0&0\\
0&1&0&0\\
0&0&2&0\\
0&0&0&2
\end{array}
\right).
 \]
According to (\ref{eq42}), we obtain $s_{11}+s_{22}\geq 0$. According to (\ref{eq43}), it holds that 
\begin{eqnarray*}
\left(\begin{array}{cccc}
1+s_{11}-\rho&0&s_{12}&0\\
0&3+s_{11}-\rho&0&s_{12}\\
s_{12}&0&2+s_{22}-2\rho&0\\
0&s_{12}&0&6+s_{22}-2\rho
\end{array}
\right)
\preceq0.
\end{eqnarray*}
Thus $\rho\geq s_{11}+3$ and $\rho\geq\frac{s_{22}+6}{2}\geq\frac{-s_{11}+6}{2}$ hold.
Finally we obtain the optimal value of {\rm (DGS1)} is ${\rm v(DGS1)}=3$, when $s_{11}=s_{22}=0$. Thus the gap is $v_{gap}=|3-\frac{7}{3}|=\frac{2}{3}>0$.
\end{example}

So we can conclude that {\rm (GS)} might have a positive Lagrangian duality gap. Consequently, only making a scaling on {\rm (GTRP)} is also not enough for closing the duality gap. So in Section 2, we first add a redundant constraint, then make a scaling to obtain {\rm (GRS)}, which has the strong Lagrangian duality property.
\begin{remark}
Consider the following problem:
\begin{eqnarray*}
{\rm (NGTRP)}~&\max\limits_{X\in\Bbb R^{n\times p}}& \frac{tr(GX^TBX)+\beta}{tr(GX^TAX)+\alpha}\\
~~~~~~~~~~~&{\rm s.t.}& X^TX=I_p.
\end{eqnarray*} 
Although {\rm (NGTRP)} is non-homogeneous, we can make the following equivalent transformation: 

Since $X^TX=I_p$, we can obtain 
\[
1=\frac{1}{p}tr(X^TX)=\frac{1}{ptr(G)}tr(GX^TX),
\]
where $\alpha\in\Bbb R$, $\beta\in\Bbb R$. Then 
\begin{eqnarray*}
&&tr(GX^TBX)+\beta=tr\left(GX^T(B+\frac{\beta}{ptr(G)}I_n)X\right),\\
&&tr(GX^TAX)+\alpha=tr\left(GX^T(A+\frac{\alpha}{ptr(G)}I_n)X\right).
\end{eqnarray*}
So {\rm (NGTRP)} is equivalent to the following {\rm (GTRP)}-like problem:
\begin{eqnarray*}
{\rm (NGTRP')}~&\max\limits_{X\in\Bbb R^{n\times p}}& \frac{tr(GX^T\tilde{B}X)}{tr(GX^T\tilde{A}X)}\\
~~~~~~~~~~~&{\rm s.t.}& X^TX=I_p,
\end{eqnarray*} 
where $\tilde{B}=B+\frac{\beta}{ptr(G)}I_n$, $\tilde{A}=A+\frac{\alpha }{ptr(G)}I_n$. So the theory about {\rm (GTRP)} can also be applied to the non-homogeneous problem {\rm (NGTRP)}.
\end{remark}

\section{Conclusion}
In this paper, we study the Lagrangian duality gap of {\rm (GTRP)}. We close the duality gap of {\rm (GTRP)} by equivalently reformulating it as {\rm (GRS)}, which is obtained by adding a redundant constraint and make a scaling. And we prove that {\rm (GRS)} has no Lagrangian duality gap by establishing a matrix S-lemma. Then we show that the following three cases: {\rm (GTRP)} itself, {\rm (GR)}, which is obtained by adding a redundant to {\rm (GTRP)}, {\rm (GS)}, which is obtained by making a scaling on {\rm (GTRP)}, might have positive Lagrangian duality gap. We will consider the following more general problem
\begin{eqnarray*}
&\max\limits_{X\in\Bbb R^{n\times p}}& \frac{tr(G_2X^TBX)}{tr(G_1X^TAX)}\\
~~~~~~~~~~~&{\rm s.t.}& X^TX=I_p
\end{eqnarray*} 
and study its properties and global algorithms in the future.
%

\end{document}